\documentclass[journal,twoside,web]{IEEETran}
\IEEEoverridecommandlockouts
\usepackage{graphics}
\usepackage{epsfig}
\usepackage{times}
\usepackage{amsmath}
\usepackage{amssymb}
\usepackage{multirow}
\usepackage{array}
\usepackage{amsfonts}
\usepackage{subcaption}

\usepackage{amsthm}
\usepackage{todonotes}
\usepackage{comment}
\usepackage{cite}
\usepackage{booktabs}
\usepackage{hyperref}

\theoremstyle{definition}
\newtheorem{assumption}{Assumption}
\newtheorem{definition}{Definition}
\theoremstyle{plain}
\newtheorem{theorem}{Theorem}

\newtheorem{lemma}{Lemma}

\newtheorem{remark}{Remark}
\newtheorem{problem}{Problem}
\theoremstyle{definition}

\title{\LARGE \bf
A Platoon Formation Framework in a Mixed Traffic Environment}

\author{A M Ishtiaque Mahbub, {\itshape{Student Member, IEEE}}, Andreas A. Malikopoulos, {\itshape{Senior Member, IEEE}}
\thanks{This research was supported by ARPAE's NEXTCAR program under the award number DE-AR0000796.}
\thanks{The authors are with the Department of Mechanical Engineering, University of Delaware, Newark, DE 19716 USA (emails: \tt\small{mahbub@udel.edu};\tt\small{andreas@udel.edu}.)}}

\begin{document}

\maketitle
\thispagestyle{empty}
\pagenumbering{arabic}

\begin{abstract}
Connected and automated vehicles (CAVs) provide the most intriguing opportunity to reduce pollution, energy consumption, and travel delays. In this paper, we address the problem of vehicle platoon formation in a traffic network with partial CAV penetration rates. We investigate the interaction between CAV and human-driven vehicle (HDV) dynamics, and provide a rigorous control framework that enables platoon formation with the HDVs by only controlling the CAVs within the network. We present a complete analytical solution of the CAV control input and the conditions under which a platoon formation is feasible. We evaluate the solution and demonstrate the efficacy of the proposed framework using simulation.

\end{abstract}
\indent


\section{Introduction}
\IEEEPARstart{T}{he} implementation of an emerging transportation system with connected and automated vehicles (CAVs) enables a novel computational framework to provide real-time control actions that optimize energy consumption and associated benefits. From a control point of view, CAVs can alleviate congestion at different traffic scenarios, reduce emission, improve fuel efficiency, and increase passenger safety \cite{Guanetti2018}. 


Significant research efforts have been reported in the literature for CAVs to improve the vehicle- and network-level performances \cite{Malikopoulos2016a,Guanetti2018}. Several research efforts have been presented for coordinating CAVs in real time at different traffic scenarios such as on-ramp merging roadways, roundabouts, speed reduction zones, signal-free intersections, and traffic corridors \cite{Mahbub2019ACC,Malikopoulos2020,mahbub2020decentralized, mahbub2020Automatica-2,mahbub2020ACC-2}. These approaches are based on the strict assumption of 100\% penetration rate of CAVs having access to perfect communication (no errors or delays), which impose limitations for real-world implementation. 
In reality, the existence of 100\% CAV market penetration is not expected before 2060 \cite{alessandrini2015automatedmixed2060}. Therefore, the need for a mathematically rigorous and tractable control framework considering the co-existence of CAVs with human-driven vehicles (HDVs), which we refer in this paper as the \emph{mixed traffic environment}, are an essential transitory step.

One of the most important research directions pertaining to the mixed traffic environment has been the development of adaptive cruise controllers \cite{zheng2017platooning}, where a CAV preceded by a group of HDVs implements a control algorithm to optimize a given objective, e.g., improvement of fuel economy, minimization of backward propagating wave \cite{hajdu2019robust}, etc.  
In a mixed traffic environment, the presence of HDVs poses significant modelling and control challenges to the CAVs due to the stochastic nature of the human-driving behavior. Although previous research efforts aimed at enhancing our understanding of improving the efficiency through coordination of CAVs in a mixed traffic environment, deriving a tractable solution still remains a challenging control problem. Several research efforts reported in the literature implemented car-following models \cite{Zhao2018CTA} to have deterministic quantification of the HDV state. 
Other research efforts have employed learning-based frameworks \cite{kreidieh2018dissipating, wu2017framework}. Although these approaches have demonstrated quite impressive performance in simulation, they might impose challenges during the trial-and-error learning process in a real-world setting.

%
%
%
%
In this paper, our research hypothesis is that we can directly control the CAVs to force the trailing HDVs to form platoons, and thus indirectly control the HDVs. In this context, we address the problem of vehicle platoon formation in mixed traffic environment by only controlling the CAVs within the network. To the best of our knowledge, such approach has not yet been reported in the literature to date.

The contribution of this paper are: (i) the development of a comprehensive framework that can aim at creating platoon formations of HDVs led by a CAV in a mixed traffic environment, and (ii) an analytical solution of the control input of CAVs (Theorems \ref{theo:1} and \ref{theo:3}), along with the conditions under which the solution is feasible (Theorems \ref{theo:2} and \ref{theo:4}).
In our exposition, we seek to establish a rigorous control framework that enables the platoon formation in a mixed environment with associated boundary conditions.

The structure of the paper is organized as follows. In Section \ref{sec:pf}, we formulate the problem of platoon formation in a mixed traffic environment. In Section \ref{sec:solution}, we provide a detailed exposition of the proposed framework, and derive analytical solution with feasibility analysis. In Section \ref{sec:sim}, we present a numerical analysis to validate the effectiveness of the proposed framework. Finally, we provide concluding remarks and future research directions in Section \ref{sec:conc}.

\section{Problem Formulation}\label{sec:pf}
We consider a CAV followed by one or multiple HDVs traveling in a single-lane roadway of length $L\in \mathbb{R}^+$. We subdivide the roadway into a \emph{buffer zone} of length $L_{b}\in \mathbb{R}^+$, inside of which the HDVs' state information is estimated (Fig. \ref{fig:platoon_zone}) (top), and a \emph{control zone} of length $L_{c}\in \mathbb{R}^+$ such that $L=L_{b}+L_{c}$, where the CAV is controlled to form a platoon with the trailing HDVs, as shown in Fig. \ref{fig:platoon_zone} (bottom). The time that a CAV enters the buffer zone, the control zone, and exits the control zone is $t^b, t^c, t^f\in \mathbb{R}^+$, respectively.

Let $\mathcal{N}=\{1,\ldots, N\}$, where $N\in \mathbb{N}$ is the total number of vehicles traveling within the buffer zone at time $t=t^c$, be the set of vehicles considered to form a platoon. 
Here, the leading vehicle indexed by $1$ is the CAV, and the rest of the trailing vehicles in $\mathcal{N}\setminus\{1\}$ are HDVs.  
We denote the set of the HDVs following the CAV to be $\mathcal{N}_{\text{HDV}}=\{2,\ldots, N\}$.
Since the HDVs do not share their local state information with any external agents, we consider the presence of a \emph{coordinator} that gathers the state information of the trailing HDVs traveling within the buffer zone. The coordinator, which can be a group of loop-detectors or comparable sensory devices, then transmits the HDV state information to the CAV at each time instance $t\in[t^b, t^c]$ using standard vehicle-to-infrastructure communication protocol. 

The objective of the CAV $1$ is to derive and implement a control input (acceleration/deceleration) at time $t^c \in \mathbb{R}^+$ so that the platoon formation with trailing HDVs in $\mathcal{N}_{\text{HDV}}$ is completed within the control zone at a given time $t^p\in (t^c, t^f]$.
\begin{figure}[b]
    \centering
    \includegraphics[scale=0.36]{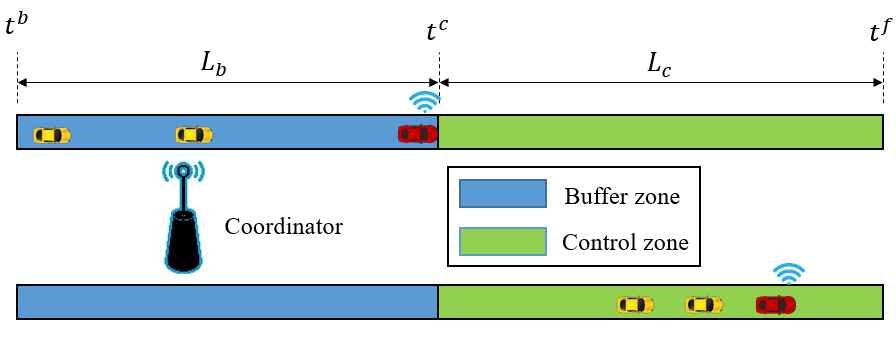}
    \caption{A CAV (red) traveling with two trailing HDVs (yellow), where the HDVs' state information is estimated (top scenario) by the coordinator within the buffer zone, and the platoon is formed (bottom scenario) by controlling the CAV inside the control zone.}
    \label{fig:platoon_zone}
\end{figure}
\begin{figure}
    \centering
    \includegraphics[scale=0.54]{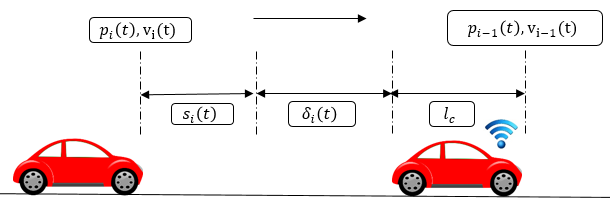}
    \caption{Predecessor-follower coupled car-following dynamic.}
    \label{fig:problem_formulation}
\end{figure}
%
In our framework, we model the longitudinal dynamics of each vehicle $i\in\mathcal{N}$ as a  double-integrator,
\begin{gather}\label{eq:dynamics_pv}
    \dot{p}_i(t) = v_i(t), \quad 
    \dot{v}_i(t) = u_i(t), \quad t\in\mathbb{R}^+,
\end{gather}
where $p_i(t)\in \mathcal{P}_i$, $v_i(t)\in \mathcal{V}_i$, and $u_i(t)\in\mathcal{U}_i$ are the position of the front bumper, speed, and control input (acceleration/deceleration) of vehicle $i\in \mathcal{N}$.
Let $\textbf{x}_{i}(t)=\left[p_{i}(t) ~ v_{i}(t)\right] ^{T}$ denote the state vector of each vehicle $i\in\mathcal{N}$, taking values in $\mathcal{X}_{i}%
=\mathcal{P}_{i}\times\mathcal{V}_{i}$. 

The speed $v_i(t)$ and control input $u_i(t)$ of each vehicle $i\in \mathcal{N}$ are subject to the following constraints,
\begin{align}\label{eq:state_control_constraints}
    0< v_{\min} \le v_i(t) &\le v_{\max}, \quad t\in\mathbb{R}^+,\nonumber\\
    u_{\min} \le u_i(t) &\le u_{\max}, \quad t\in\mathbb{R}^+,
\end{align}
where $v_{\min}$ and $v_{\max}$ are the minimum and maximum allowable speed of the considered roadway, respectively, and $u_{\min}$ and $u_{\max}$ are the minimum and maximum control input of all vehicles $i\in \mathcal{N}$, respectively.

The dynamics \eqref{eq:dynamics_pv} of each vehicle $i\in\mathcal{N}$ can take different forms based on the consideration of connectivity and automation.
For the CAV $1\in\mathcal{N}$, the control input $u_1(t)$ can be derived and implemented within the control zone. We introduce and discuss the structure of the control zone in detail in Section \ref{sec:solution}.
%
To model the HDV dynamics, we need the following definitions.
\begin{definition}
The dynamic following spacing $s_i(t)$ between two consecutive vehicles $i \text{ and }(i-1)\in\mathcal{N}$ is,
\begin{equation}\label{eq:s_i}
   { s_i(t)= \rho_i\cdot v_i(t)+ s_0,}
\end{equation}
where  $\rho_i$ denotes a desired time gap that each HDV $i\in\mathcal{N}_{\text{HDV}}$ maintains while following the preceding vehicle, and $s_0$ is the standstill distance denoting the minimum bumper-to-bumper gap at stop.
\end{definition}

\begin{definition}\label{def:2}
The \textit{platoon gap} $\delta_i(t)$ is the difference between the bumper-to-bumper inter-vehicle spacing and the dynamic following spacing {$s_i(t)$} (see Fig. \ref{fig:problem_formulation}) between two consecutive vehicles $i \text{ and }(i-1)\in\mathcal{N}$, i.e.,
\begin{equation}\label{eq:delta}
    \delta_i(t)=p_{i-1}(t)- p_i(t)-s_i(t)-l_c,
\end{equation}
where $l_c$ is the length of each vehicle $i\in\mathcal{N}$.
\end{definition}
In this paper, we adopt the optimal velocity car-following model \cite{bando1995dynamical}, to define the predecessor-follower coupled dynamics (see Fig. \ref{fig:problem_formulation}) of each HDV $i\in \mathcal{N}_{\text{HDV}}$ as follows,
\begin{gather}\label{eq:hdv_dynamics}
    {{u}_i(t) = \alpha (V_i(\delta_i(t-\eta_i),s_i(t-\eta_i)) -v_i(t-\eta_i)),}
\end{gather}
where $\alpha$ denotes the control gain representing the driver's sensitivity coefficient, {$\eta_i$ is the driver's perception delay with a known upper bound $\bar{\eta}$}, and $V_i(\delta_i(t),s_i(t))$ denotes the equilibrium speed-spacing function,
\begin{gather}
V_i(\delta_i(t),s_i(t))=
\begin{array}
[c]{ll}%
 {\frac{v_{\max}}{2}(\tanh(\delta_i(t))}{+\tanh(s_i(t))).}
\end{array}
\label{eq:V(s)}
\end{gather}

\begin{remark}\label{rem:ovm}
Based on \eqref{eq:V(s)}, the driving behavior of each HDV $i\in\mathcal{N}_{\text{HDV}}$ depends on two different modes; (a) \emph{decoupled free-flow mode}: when $\delta_i(t)>0$, each HDV converges to the maximum allowable speed $v_{\max}$, and cruises through the roadway decoupled from the state of the preceding vehicle, 
and (b) \emph{coupled following mode}: when $ \delta_i(t)\le 0$, the HDV dynamics becomes coupled with the state of the preceding vehicle $(i-1)\in\mathcal{N}$, and $v_i(t)$ converges to $v_{i-1}(t)$.
Note that, if there is no preceding vehicle, we set $\delta_i(t)=\infty$ that activates the decoupled free-flow mode, which results in $v_i(t)$ converging to $v_{\max}$.
\end{remark}
\begin{remark}\label{rem:platoon-stability}
The car-following model \eqref{eq:hdv_dynamics} is platoon-stable, i.e., bounded speed fluctuation between two consecutive vehicles in coupled following mode decays exponentially as time progresses \cite{bando1995dynamical}.
\end{remark}


We now provide the following definitions that are necessary for the formulation of our proposed platoon formation framework.
\begin{definition}\label{def:1}
The information set $\mathcal{I}_1(t)$ of the CAV $1\in\mathcal{N}$ has the following structure,
\begin{equation}
    \mathcal{I}_1(t) = \{\textbf{x}_1(t), \textbf{x}_{2:N}(t)\},
        \quad \quad t\in[t^b,t^c],
\end{equation}
where $\textbf{x}_{2:N}(t)=[\textbf{x}_2(t),\ldots, \textbf{x}_N(t)]^T$.
\end{definition}

\begin{definition}\label{def:3}
The steady-state traffic flow between two consecutive vehicles $i\text{ and }(i-1)\in\mathcal{N}$ are established if the platoon gap $\delta_i(t)$ does not vary with time, and {speed fluctuation $\Delta v_i(t):=v_i(t)-v_{i-1}(t)$ is zero} \cite{rothery1992car}, i.e., 
\begin{gather}
 \delta_i(t)= c_i,~c_i\in \mathbb{R}, \text{ and }{\Delta v_i(t) =0}.
\end{gather}
\end{definition}


We now formalize the problem of platoon formation in mixed environment addressed in the paper as follows.
\begin{problem}\label{prob:1}
    Given the information set $\mathcal{I}_1(t)$ at time $t=t^c$, the objective of the CAV $1\in\mathcal{N}$ is to derive the control input $u_1(t)$ so that the HDVs in ${N}_{HDV}$ are forced to form a platoon at some time $t^p\in (t^c, t^f]$ within the control zone while the following conditions hold,
    \begin{align}
    v_i(t) = v_{eq},~  &\delta_i(t)=c_i,~c_i\le 0, \quad \forall t\ge t^p,~ \forall i\in\mathcal{N}, \nonumber\\
   &\text{ subject to: }{ \eqref{eq:state_control_constraints},~ p_1(t^p)\le L_c,} \label{eq:platoon_cond_1}
    \end{align}
    where, $v_{eq}$ denotes the equilibrium platoon speed.
\end{problem}
\begin{remark}\label{rem:delta>0}
In our problem formulation, we impose the restriction that at $t=t^c$, there exists at least one HDV $i\in\mathcal{N}_{\text{HDV}}$ such that $\delta_i(t^c)>0$. To simplify the formulation and without loss of generality, we consider that $\delta_N(t^c)>0$. This ensures that we do not have the trivial case where the group of vehicles in $\mathcal{N}$ has already formed a platoon at $t=t^c$. 
\end{remark}
In the modelling framework presented above, we impose the following assumption.
\begin{assumption}\label{assum:1}
The CAV is on a decoupled free-flow mode (Remark \ref{rem:ovm}) while all vehicles have reached steady-state traffic flow (Definition \ref{def:3}) within $[t^b, t^c]$.
\end{assumption}

\begin{remark}\label{rem:control_zone}
{We restrict the control of the CAV $1$ only within the control zone so that we have a finite control horizon $[t^c,t^f]$. Outside the control zone, the CAV dynamics follows the car-following model in \eqref{eq:hdv_dynamics}.}
\end{remark}

\begin{lemma}\label{lem:1}
For each vehicle $i\in\mathcal{N}$, $v_i(t^c)=v_{\max}$.
\end{lemma} 
\begin{proof}
{Since the control input $u_1(t)$ of the uncontrolled CAV $1\in\mathcal{N}$ is determined by \eqref{eq:hdv_dynamics} outside the control zone (Remark \ref{rem:control_zone}), and due to the fact that any vehicle following the dynamics in \eqref{eq:hdv_dynamics} converges to the maximum speed $v_{max}$ without the presence of a preceding vehicle (Remark \ref{rem:ovm}), $v_1(t)$ converges to $v_{\max}$.}

For a HDV $i\in\mathcal{N}_{\text{HDV}}$ traveling under the steady-state traffic flow condition (Assumption \ref{assum:1}), $\delta_i(t)$ does not vary with time. This implies that each HDV $i$ either travels with decoupled free-flow mode with  $v_i(t)=v_{\max}$, or with coupled following mode with $v_i(t)=v_{i-1}(t)=v_{\max}$.
\end{proof} 
In what follows, first, we address Problem \ref{prob:1} considering only two vehicles, i.e., $N=2$, and then generalize the analysis for multiple HDVs, i.e., $N>2$.
%
\section{Vehicle Platoon Formation Framework}\label{sec:solution}
{For $N=2$, CAV $1\in\mathcal{N}$ is trailed by HDV $2\in\mathcal{N}_{\text{HDV}}$.} The information set $\mathcal{I}_1(t^c)$ includes $v_1(t^c)=v_2(t^c)=v_{\max}$ (Lemma \ref{lem:1}), and $\delta_2(t^c)>0$ (Remark \ref{rem:delta>0}).
\subsection{Control input of CAV $1\in\mathcal{N}$}
The following result characterizes the control structure of CAV $1\in\mathcal{N}$ for the platoon formation framework.
\begin{lemma}\label{lem:2}
For a CAV $1\in\mathcal{N}$ travelling with a trailing HDV $2\in\mathcal{N}_{\text{HDV}}$,
        (i) a platoon formation does not occur when $u_1(t)\ge0$ for all $t\in [t^c,t^f]$, and
        (ii) {a platoon formation occurs with an appropriate control zone of length $L_c$ when $u_1(t)<0$ for all $t\in [t^{c},t^{a}],~ t^c< t^a < t^f$}.
\end{lemma}
\begin{proof}

Part (i):  For $u_1(t)\ge 0$ for all $t \ge t^c$, we have $\delta_2(t)>0$ for all $t \ge t^c$, which implies that according to \eqref{eq:platoon_cond_1}, no platoon formation will occur.

Part (ii): {For $u_1(t)<0$ within an arbitrary time horizon $[t^c,t^a], t^c< t^{a} < t^f $, we have $v_1(t^a) < v_1(t^c)$. Since ${v_1(t^c)}=v_2(t^c)=v_{\max}$ (Lemma \ref{lem:1}), we have $v_2(t^a)>v_1(t^a)$. This implies that $\delta_2(t)$ decreases for all $t\ge t^a$. As time $t$ progresses, given an appropriate control zone of length $L_c$, we have $\delta_2(t)\rightarrow 0$, which guarantees a platoon formation.}
\end{proof}
%
When the CAV $1\in\mathcal{N}$ applies a control input $u_1(t),~t\in[t^c, t^p]$ based on Lemma \ref{lem:2} to form a platoon with the HDV $2\in\mathcal{N}_{\text{HDV}}$ at time $t=t^p$, two sequential steps take place, namely, (i) the \emph{platoon transition} step, where the HDV $2$ transitions from the decoupled free-flow mode to the coupled following mode at time $t=t^s, ~t^c< t^s< t^p$ such that $\delta_2(t^s)=0$, and (ii) the \emph{platoon stabilization} step, where $v_2(t)$ converges to $v_1(t)$ at time $t=t^p$ such that \eqref{eq:platoon_cond_1} is satisfied, and the platoon becomes stable.
\begin{definition}\label{def:tau}
The platoon transition duration $\tau^t$ is the time required for the completion of the platoon transition step, i.e., $\tau^t=t^s-t^c$, and the platoon stabilization duration $\tau^s$ is the time required for the completion of the platoon stabilization step, i.e., $\tau^s=t^p-t^s$. Hence, we have $t^p=t^c+\tau^{t}+\tau^{s}$.
\end{definition}
\begin{remark}\label{rem:tau_ps}
{
The platoon stabilization duration is $\tau^s=\eta_i+\tau^r$, where $\eta_i$ is the perception delay of HDV $i\in\mathcal{N}_{\text{HDV}}$, and $\tau^r$ is the response time of $\eqref{eq:hdv_dynamics}$ which depends on the driver’s sensitivity coefficient $\alpha$, maximum allowable speed fluctuation, and the choice of equilibrium speed-spacing function in \eqref{eq:V(s)}, and can be computed using stability analysis presented in \cite{bando1995dynamical,wilson2011car,zhang2021improved}. {Note that, for $N\ge 2$,  additional nonlinearities may impact the computation of $\tau^s$.} 
In our formulation, we incorporate the upper bound of the perception delay $\bar{\eta}$ to achieve robustness such that $\tau^s=\bar{\eta}+\tau^r$, and consider that $\tau^r$ is given a priori. Thus we focus only on the analysis of the platoon transition time $\tau^t$.}
\end{remark} 
Using Lemma \ref{lem:2}, we construct the structure of the control input $u_1(t)$ for the CAV $1\in\mathcal{N}$ for generating a platoon with the trailing HDV $2\in\mathcal{N}$ at time $t^p\in(t^c, t^f]$,
\begin{gather}\label{eq:control_structure}
 u_{1}(t) = \left\{
\begin{array}
[c]{ll}%
         u_p, ~u_p\in[u_{\min},0), & t\in[t^c, t^s],\\
        0, ~  & t\in( t^s,t^f].
\end{array}
\right.
\end{gather}
According to \eqref{eq:control_structure}, the realization of the control input $u_1(t)$ of the CAV $1\in\mathcal{N}$, which is $u_p\in(0,u_{\min}]$ in $t\in[t^c, t^s]$, yields a linearly decreasing $v_1(t)$ in $t\in[t^c, t^s]$.

The following result provides the unconstrained relation between the platoon transition duration $\tau^t$ and CAV control input parameter $u_p$.

\begin{theorem}\label{theo:1}
For a CAV $1\in\mathcal{N}$ and a trailing HDV $2\in\mathcal{N}_{\text{HDV}}$, there exists an unconstrained control input parameter $u_p$ in \eqref{eq:control_structure} such that a vehicle platoon can be formed with HDV $2\in\mathcal{N}$ at time $t=t^p$ according to the following condition,
\begin{equation}\label{eq:u_t_relation_1}
    2\delta_2(t^c) + u_p\cdot(\tau^t)^2 = 0.
\end{equation}
\end{theorem} 
\begin{proof}
{At $t^s=t^c+\tau^t$, we require $\delta_2(t^s)=0$ implying $p_1(t^s)-p_2(t^s)= s_2(t^s)+l_c$, which we expand as follows. Using \eqref{eq:dynamics_pv} at time $t^s=t^c+\tau^t$, we have $p_{1}(t^c+\tau^t) = p_{1}(t^c) + v_{1}(t^c)\cdot \tau^t+ \frac{1}{2} u_p\cdot (\tau^t)^2$.
Based on Lemma \ref{lem:1}, $v_1(t^c)=v_{max}$. For HDV $2\in\mathcal{N}_{\text{HDV}}$, $\delta_2(t)>0$ (Remark \ref{rem:delta>0}) until the platoon transition step at time $t=t^s$. This implies, that HDV $2$ travels with decoupled free-flow mode as in \eqref{eq:hdv_dynamics}, and $v_2(t)=v_{\max}$ for all $t\in [t^c, t^s]$ (Lemma \ref{lem:1}). Using \eqref{eq:hdv_dynamics} for HDV $2$ at time $t^s = t^c+\tau^t$, we have, $v_2(t^c+\tau^t) = v_2(t^c)=v_{max}$ and $p_2(t^c+\tau^t) = p_2(t^c) + v_2(t^c)\cdot \tau^t$. Substituting the last equation into \eqref{eq:s_i}, we have $s_2(t^s)=s_2(t^c)$, and hence 
$p_1(t^c) + v_1(t^c)\cdot \tau^t+ \frac{1}{2} u_p \cdot (\tau^t)^2- p_2(t^c) - v_2(t^c)\cdot \tau^t = s_2(t^c)+l_c. \label{eq:theo1:6}$
Simplifying using \eqref{eq:delta}, the result follows.}
\end{proof}

\begin{remark}
From \eqref{eq:u_t_relation_1}, as $u_p\rightarrow 0$, we have $\tau^t \rightarrow \infty$, which implies that platoon formation will never occur. If $u_p>0$, then \eqref{eq:u_t_relation_1} yields an infeasible $\tau^t$. Therefore, $u_p$ has to be strictly negative for platoon formation. Note, from \eqref{eq:u_t_relation_1}, for $\delta_2(t)>0$ and $t\in \mathbb{R}^+$, we have $u_p<0$.
\end{remark}
\subsection{Feasibility of the platoon formation time, $t^p$}
In Theorem \ref{theo:1}, we do not explicitly incorporate the state and control constraints in \eqref{eq:state_control_constraints}, and the terminal constraint in \eqref{eq:platoon_cond_1}. For a given platoon formation time $t^p$, the corresponding control input derived from \eqref{eq:u_t_relation_1} can violate constraints in \eqref{eq:state_control_constraints}. In what follows, we present Lemmas \ref{lem:3} and \ref{lem:4} that provide a feasible region of $\tau^t$ that yields an admissible control input parameter $u_p$ in \eqref{eq:u_t_relation_1}.
\begin{lemma}\label{lem:3}
For CAV $1\in\mathcal{N}$, the platoon transition duration $\tau^t$ subject to the state and control constraints in \eqref{eq:state_control_constraints} is feasible if the following condition holds,
\begin{gather}
   \tau^t \ge \max\bigg\{ \bigg(\frac{-2\delta_2(t^c)}{u_{\min}}\bigg)^{\frac{1}{2}}, \frac{2\delta_2(t^c)}{v_1(t^c)-v_{\min}} \bigg\}. \label{eq:lem:3}
\end{gather}
\end{lemma}
\begin{proof}
Suppose that, for CAV $1\in\mathcal{N}$, $u_p=u_{\min}$ yields a corresponding platoon transition duration $\tau^{t_1}$. From \eqref{eq:u_t_relation_1}, we have $(\tau^{t_1})^2 = \frac{-2\delta_2(t)}{u_{\min}}$. 
Therefore, for any $\tau^t$ to be feasible such that {$u_p\in [u_{\min},0)$}, we require $\tau^t\ge \tau^{t_1}$, which yields the inequality with the first term in \eqref{eq:lem:3}.

Now, suppose that for CAV $1\in\mathcal{N}$, a platoon transition duration $\tau^{t}$ has associated control input parameter $u_p$ derived from \eqref{eq:u_t_relation_1}.
Using \eqref{eq:dynamics_pv}, we have,
   $v_1(t^c+\tau^{t})=v_1(t^c)+u_p \cdot \tau^{t}.$
Since $u_p\in[u_{\min},0)$, we require that $v_1(t^c+\tau^t)\ge v_{\min}$ to satisfy the state constraint in \eqref{eq:state_control_constraints}. Substituting $v_1(t^c+\tau^t)$ in the above inequality, we get $ u_p\cdot \tau^t \ge v_{\min}-v_1(t^c)$. Finally, substituting $u_p$ from \eqref{eq:u_t_relation_1} in the above equation yields the inequality with the second term in \eqref{eq:lem:3}. 

Finally, since both above inequalities yield lower bounds on $\tau^t$, we simply take their maximum and get \eqref{eq:lem:3}.
\end{proof}
\begin{remark}\label{rem:v_min}
The minimum speed value $v_{\min}$ in \eqref{eq:lem:3} indicates the allowable speed perturbation during the platoon stabilization step. Hence, $v_{\min}$ should be selected appropriately to ensure local stability of the platoon \cite{wilson2011car, zhang2021improved}.
\end{remark}
\begin{lemma}\label{lem:4}
For the CAV $1\in\mathcal{N}$ subject to the control input $\eqref{eq:control_structure}$, the following condition must hold in order to complete platoon formation at time $t=t^p$ within the control zone of length $L_c$,
\begin{equation}
    \tau^t \le  \frac{\phi_1}{2}+\frac{\sqrt{\phi_1^2+4\phi_2}}{2},\label{eq:lem:4}
\end{equation}
where, $\phi_1:=\frac{L_c+\delta_2(t^c)-v_1(t^c)\cdot (\tau^r+\bar{\eta})}{v_1(t^c)}$, and $\phi_2:=\frac{2\delta_2(t^c)\cdot (\tau^r+\bar{\eta})}{v_1(t^c)}$.
\end{lemma}
\begin{proof}
{Suppose that, for CAV $1\in\mathcal{N}$, $p_1(t^p)-p_1(t^c)\le L_c$. Using \eqref{eq:control_structure}, $p_1(t^p)=p_1(t^c+\tau^t)+v_1(t^s)\cdot \tau^s$, which yields}
\begin{gather}
    p_1(t^c+\tau^t)-p_1(t^c)+v_1(t^s)\cdot \tau^s \le L_c.\label{eq:lem4_1}
\end{gather}
From \eqref{eq:dynamics_pv} and \eqref{eq:control_structure}, we have, $p_1(t^c+\tau^t)=p_1(t^c)+v_1(t^c)\cdot \tau^t + \frac{1}{2}u_p (\tau^t)^2$, and $v_1(t^s)=v_1(t^c)+u_p \tau^t$. {Substituting $p_1(t^c+\tau^t)$, $v_1(t^s)$ into \eqref{eq:lem4_1}, $\tau^s=\tau^r+\bar{\eta}$ from Remark \ref{rem:tau_ps},
and using \eqref{eq:control_structure}-\eqref{eq:u_t_relation_1}, we have, $\tau^t - \frac{2\delta_2(t^c)\cdot \tau^s}{v_1(t^c)\cdot \tau^t}\le \frac{L_c+\delta_2(t^c)-v_1(t^c)\cdot \tau^s}{v_1(t^c)}$. Simplifying and letting $\phi_1=\frac{L_c+\delta_2(t^c)-v_1(t^c)\cdot (\tau^r+\bar{\eta})}{v_1(t^c)}$}, and $\phi_2=\frac{2\delta_2(t^c)\cdot (\tau^r+\bar{\eta})}{v_1(t^c)}$, the above equation yields a quadratic inequality $(\tau^t)^2-\phi_1 \tau^t - \phi_2 \le 0$, solving which yields \eqref{eq:lem:4}.
\end{proof}
The following result provides the condition under which for a given platoon formation time $t^p$ and platoon stabilization duration $\tau^s$, the corresponding platoon transition duration $\tau^t$ is feasible.
\begin{theorem}\label{theo:2}
For a CAV $1\in\mathcal{N}$ to complete the platoon transition step with its following HDV $2\in\mathcal{N}_{\text{HDV}}$ with control input $u_1(t)=u_p,~t\in[t^c,t^s]$ within the control zone of length $L_c$, a platoon transition duration $\tau^t$ is feasible if,
\begin{align}
   \max\bigg\{ \bigg(\frac{-2\delta_2(t^c)}{u_{\min}}\bigg)^{\frac{1}{2}},& \frac{2\delta_2(t^c)}{v_1(t^c)-v_{\min}} \bigg\}\le \tau^t \nonumber\\ 
   &\le  \frac{\phi_1+\sqrt{\phi_1^2+4\phi_2}}{2}, \label{eq:theo_2}
\end{align}
holds.
\end{theorem}

\begin{proof}
The proof follows directly from Lemmas \ref{lem:3} and \ref{lem:4}.
\end{proof}
\subsection{Extension of the Analysis for $N>2$}\label{sec:extension}
{For $N>2$, the CAV $1\in\mathcal{N}$ trailed by multiple HDVs $j\in\mathcal{N}_{\text{HDV}}$ and given $\mathcal{I}_1(t^c)$, we have the following conditions, $v_1(t^c)=v_j(t^c)=v_{\max}$ for all $j\in\mathcal{N}_{\text{HDV}}$ (Lemma \ref{lem:1}), and there exists $j\in\mathcal{N}_{\text{HDV}}$ such that $\delta_j(t^c)>0$ (Remark \ref{rem:delta>0}).}
{
\begin{definition}\label{def:cumDelta}
For a CAV $1\in\mathcal{N}$ followed by $N \in \mathcal{N_{\text{HDV}}}$ HDVs, the cumulative platoon gap $\Delta(t)$ at time $t\in[t^c,t^f]$ is, 
\begin{gather}
    \Delta(t) = p_1(t)-p_{N}(t)-\sum_{j=2}^{N} (s_j(t)+ l_c).
\end{gather}
\end{definition}}
%
%
In what follows, we extend the analysis presented in Theorems \ref{theo:1} and \ref{theo:2}, and derive results that enables platoon formation considering multiple trailing HDVs, i.e., $N>2$. The  following  theorem  provides  the  unconstrained  relation between  the  platoon  transition  duration $\tau^t$ and CAV control input parameter $u_p$ for $N>2$.
\begin{theorem}\label{theo:3}
{For a CAV $1\in\mathcal{N}$ followed by $N$ HDVs $j\in\mathcal{N}_{\text{HDV}}$, there exists an unconstrained control input parameter $u_p$ in \eqref{eq:control_structure} such that a vehicle platoon can be formed with HDVs $j\in\mathcal{N}$ at time $t=t^p$ according to the following relation,
\begin{equation}\label{eq:theo:3}
    2\Delta(t^c) + u_p\cdot(\tau^t)^2-2  u_p \tau^t\sum_{j=2}^{N-1} \rho_j = 0.
\end{equation}}
\end{theorem}
\begin{proof}
{At $t^s=t^c+\tau^t$, we require $p_1(t^s)-p_{N}(t^s)= \sum_{j=2}^{N-1} s_j(t^s)+ s_{N}(t^s)+\sum_{j=2}^{N} l_c$. Using \eqref{eq:dynamics_pv}, we have $p_1(t^c+\tau^t)=p_1(t^c)+v_1(t^c)\cdot \tau^t+\frac{1}{2}u_p\cdot (\tau^t)^2$, and $p_{N}(t^c+\tau^t)=p_{N}(t^c)+v_{N}(t^c)\cdot \tau^t$. Substituting $p_1(t^c+\tau^t)$ and $p_N(t^c+\tau^t)$ into the last equation and simplifying, we have,
\begin{equation}\label{eq:theo_3_1}
    p_1(t^c)-p_N(t^c)-\sum_{j=2}^{N-1} s_j(t^s)- s_{N}(t^s)-\sum_{j=2}^{N} l_c=-\frac{1}{2}u_p\cdot (\tau^t)^2.
\end{equation} 
At $t=t^s$, we have $s_{N}(t^s)=s_{N}(t^c)$, $s_j(t^s)=\rho_j v_1(t^s)+s_0$ and $s_j(t^c)=\rho_j v_j(t^c)+s_0$ for $j=2,\ldots, N-1$. With $v_1(t^s)=v_1(t^c)+u_p\tau^t$, we have $s_j(t^s)=s_j(t^c)+\rho_j u_p \tau^t$, $j=2,\ldots, N-1$. Using the last equations in \eqref{eq:theo_3_1}, we have $p_1(t^c)-p_N(t^c)-\sum_{j=2}^{N} s_j(t^c)- \sum_{j=2}^{N-1} \rho_j u_p \tau^t -\sum_{j=2}^N l_c=-\frac{1}{2}u_p\cdot (\tau^t)^2$, and using Definition \ref{def:cumDelta}, the results follows.}
\end{proof}
For $N>2$, the following result provides the condition under which for a given platoon formation time $t^p$ and platoon stabilization duration $\tau^s$, the corresponding platoon transition duration $\tau^t$ in Theorem \ref{theo:3} is feasible.
\begin{theorem}\label{theo:4}
For a CAV $1\in\mathcal{N}$ to complete the platoon transition step with its following $N$ HDVs $j\in\mathcal{N}_{\text{HDV}}$ with control input $u_1(t)=u_p,~t\in[t^c,t^s]$ within the control zone of length $L_c$, a platoon transition duration $\tau^t$ is feasible if,
{
\begin{align}\label{eq:theo_4}
   \max\bigg\{ \bigg( C_1+\sqrt{C_1^2-\frac{2\Delta(t^c)}{u_{\min}}} \bigg), 2C_1+\frac{2\Delta(t^c)}{v_1(t^c)-v_{\min}} \bigg\}\\ \le \tau^t 
   \le  \frac{\phi_3+\sqrt{\phi_3^2+4\phi_4}}{2},\nonumber
\end{align}
holds,
where $C_1:=\sum_{j=2}^{N-1} \rho_j$, $C_2:=L_c-v_1(t^c)\cdot\tau^s$, $\phi_3:=\frac{2C_1v_1(t^c)+\Delta(t^c)+C_2}{v_1(t^c)}$, and $\phi_4:=\frac{2\Delta(t^c)\cdot \tau^s-2C_1C_2}{v_1(t^c)}$.}
\end{theorem}
\begin{proof}
{Suppose that, for CAV $1\in\mathcal{N}$, $u_p=u_{\min}$ yields a corresponding platoon transition duration $\tau^{t_1}$. From \eqref{eq:theo:3}, we have $(\tau^{t_1}) =  C_1+\sqrt{C_1^2-\frac{2\Delta(t^c)}{u_{\min}}}$, where $C_1:=\sum_{j=2}^{N-1} \rho_j$.
Therefore, for any $\tau^t$ to be feasible such that $u_p\in [u_{\min},0)$, we require $\tau^t\ge \tau^{t_1}$, which yields the inequality with the first term in \eqref{eq:theo_4}.
\newline
Now for the second inequality term, since $u_p\in[u_{\min},0)$, we require that $v_1(t^c+\tau^t)\ge v_{\min}$ to satisfy the state constraint in \eqref{eq:state_control_constraints}. Using \eqref{eq:dynamics_pv}, we have, $v_1(t^c+\tau^{t})=v_1(t^c)+u_p \cdot \tau^{t}$.
 Substituting $v_1(t^c+\tau^t)$ in the above inequality, we get $ u_p\cdot \tau^t \ge v_{\min}-v_1(t^c)$. Substituting $u_p$ from \eqref{eq:theo:3} and simplifying, we have the inequality with the second term in \eqref{eq:theo_4}. Since both left-hand side inequalities mentioned above give lower bounds on $\tau^t$, we simply take their maximum and get the left inequality of \eqref{eq:theo_4}.
\newline
Finally, using the result of Theorem \ref{theo:3} and following similar steps to those in the proofs of Lemma \ref{lem:4}, we derive the right inequality of \eqref{eq:theo_4}.}
\end{proof}
%
\section{Numerical Example}\label{sec:sim}
To demonstrate the performance of the proposed platoon formation framework, we present the simulation considering $\mathcal{N}=\{1,2,3\}$ consisting of a CAV $1$ followed by two HDVs $2$ and $3$, using numerical simulation in MATLAB R2020b. 
%
{For a desired platoon formation time $t^p=47.2$ s and a given platoon stabilization duration $\tau^s=5$ s, $\tau^t=42.2$ s is feasible according to Theorem \ref{theo:4}, and we use Theorem \ref{theo:3} to compute the corresponding control input $u_p$ for CAV $1$. The headway trajectories of HDVs $2$ and $3$ converge to the equilibrium value and remain time invariant for all $t> t^p$, as shown in Fig. \ref{fig:platoon_result} (bottom). Since the conditions in \eqref{eq:platoon_cond_1} are satisfied for all $t\ge t^p$, the platoon formation is completed at time $t=t^p$ s as indicated by the position trajectories shown in Fig. \ref{fig:platoon_result} (top).}
\begin{figure}
    \centering
    \includegraphics[width=3.8in]{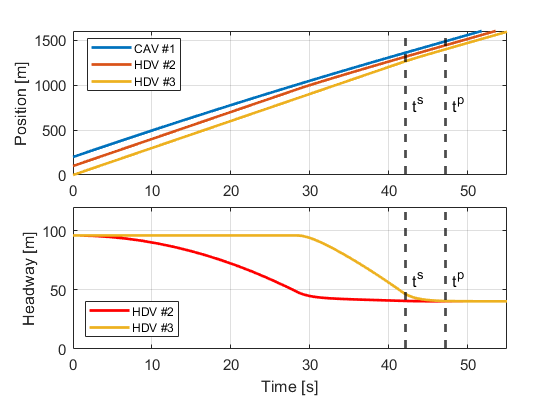}
    \caption{Platoon formation for $N=3$, where the position (top) and headway (bottom) of the vehicles are illustrated.}
    \label{fig:platoon_result}
\end{figure}

{In Fig. \ref{fig:platoon_robustness}, we show the robustness of the proposed framework in terms of \emph{platoon formation deviation} representing the percentage deviation of the actual platoon formation time $t^{ap}$ from the desired platoon formation time $t^p$, i.e., $\frac{t^{ap}-t^{p}}{t^p}\times100 [\%]$, for $N= 2,~3 \text{ and }4$. Here, positive platoon formation deviation indicates delayed platoon formation in actual simulation, and conversely, negative deviation indicates platoon formation before $t^p$. Figure \ref{fig:platoon_robustness}(a) shows that the platoon is formed within $2.5\%$ deviation for all admissible $\tau^t$, where the higher $\tau^t$ values minimizes delayed platoon formation instances. The robustness of the framework under different perception delay $\eta_i\in[0,1]$ is showed in Fig. \ref{fig:platoon_robustness}(b). Since the platoon formation deviations are mostly non-positive, the conservative consideration of $\bar{\eta}$ guarantees platoon formation within the desired platoon formation time $t^p$. 

Finally, we consider the variation of two car-following parameters, namely the desired time gap $\rho_i\in[0.5,1.5]$ and driver's sensitivity coefficient $\alpha\in[1,2]$, to investigate the performance of the proposed framework under random human driving behavior based on \eqref{eq:hdv_dynamics}, as shown in Fig. \ref{fig:platoon_robustness}(c) and (d), respectively. The proposed framework is mostly robust against variation of $\rho_i$, and shows delayed platoon formation only near the maximum value of $\rho_i$. In contrast, the proposed framework shows delayed platoon formation with $<3\%$ deviation for variation of $\alpha$. Note that, since $\tau^s$ is dependent on $\alpha$, the appropriate computation of $\tau^s$ can minimize the platoon formation deviation with varying $\alpha$.}

{Supplementary videos of the simulation
and experimental results of the proposed framework as well as
the parameters used for the simulation results can be found at: \href{https://sites.google.com/udel.edu/platoonformation}{https://sites.google.com/udel.edu/platoonformation}}.
\section{Discussion and concluding Remarks}\label{sec:conc}
In this paper, we presented a framework for platoon formation under a mixed traffic environment, where a leading CAV derives and implements its control input to force the following HDVs to form a platoon. Using a predefined car-following model, we
provided a complete, analytical solution of the CAV control input intended for the platoon formation. We also provided
a detailed analysis of the platoon formation framework, and provided conditions under which a feasible platoon formation time exists. Finally, we presented numerical example to validate the robustness of our proposed framework.

{A direction for future research should extend the proposed framework to make it agnostic to additional car-following models. Ongoing research considers the notion of optimality to derive energy- or time-optimal platoon formation framework under relaxed assumption on the steady-state traffic flow. }
\begin{figure}
    \centering
    \includegraphics[width=3.8in]{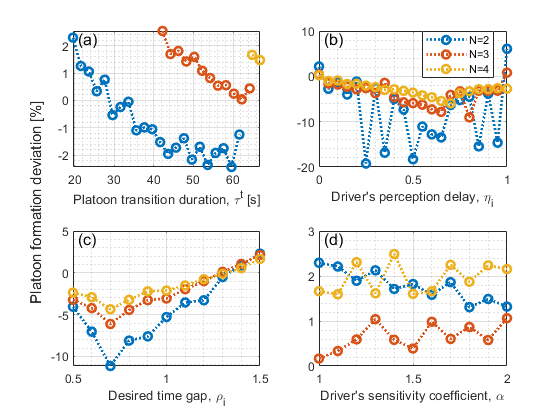}
    \caption{{Percentage deviation of actual platoon formation time vs. desired platoon formation time for $N=2,~3$ and $4$ under the consideration of different platoon transition duration $\tau^t$ (subfigure (a)), perception delay $\eta_i$ (subfigure (b)), and car-following model parameters $\rho_i$ and $\alpha$ (subfigure (c)-(d)), respectively.}}
    \label{fig:platoon_robustness}
\end{figure}
\bibliographystyle{IEEEtran}
\bibliography{IDS_Publications_06112021,acc_pt_vd_ref, platoon}

%

%

\end{document}